\documentclass[5p,preprint]{elsarticle}

\usepackage[utf8]{inputenc}
\usepackage[T1]{fontenc}

\usepackage{amsmath}
\usepackage{amssymb}
\usepackage{amsfonts}
\usepackage{amsthm}
\usepackage{enumerate}
\usepackage{dsfont}
\usepackage{xcolor}
\usepackage{color, colortbl}
\definecolor{hage}{rgb}{0.4,0.6,1}

\definecolor{Gray}{gray}{0.9}

\usepackage{tikz}
\usepackage{pgfplots}
\usepackage{caption}
\usepackage{url}
\usepackage{subcaption}
\usetikzlibrary{fit}
\usetikzlibrary{calc}
\pgfdeclarelayer{background}
\pgfsetlayers{background,main}
\colorlet{inbox}{lightgray!20}
\colorlet{outbox}{lightgray!50}

\makeatletter
\renewcommand*\env@matrix[1][*\c@MaxMatrixCols c]{%
  \hskip -\arraycolsep
  \let\@ifnextchar\new@ifnextchar
  \array{#1}}

\usepackage{mathdots}

\theoremstyle{definition}
\newtheorem{definition}{Definition}[section]

\theoremstyle{theorem}
\newtheorem{theorem}[definition]{Theorem}
\newtheorem{lemma}[definition]{Lemma}

\newtheorem{proposition}[definition]{Proposition}

\newcommand{\RN}[1]{\uppercase\expandafter{\romannumeral#1}}
\newcommand{\eps}{\varepsilon}

\newcommand{\R}{{\mathbb{R}}}

\newcommand{\vp}{\varphi}

\newcommand{\setdef}[2]{\left\{\ #1\ \left|\ \vphantom{#1} #2\ \right.\right\}}
\newcommand{\ddt}{\tfrac{\text{\normalfont d}}{\text{\normalfont d}t}}

{\left(\begin{smallmatrix}}
{\end{smallmatrix}\right)}

{\left[\begin{smallmatrix}}
{\end{smallmatrix}\right]}

\begin{document}

\begin{frontmatter}

\title{Feedback control of the COVID-19 pandemic with\\ guaranteed non-exceeding ICU capacity}

\author[Paderborn]{Thomas Berger}\ead{thomas.berger@math.upb.de}

\address[Paderborn]{Universit\"at Paderborn, Institut f\"ur Mathematik, Warburger Str.~100, 33098~Paderborn, Germany}


\begin{abstract}
In this paper we investigate feedback control techniques for the COVID-19 pandemic which are able to guarantee that the capacity of available intensive care unit beds is not exceeded. The control signal models the social distancing policies enacted by local policy makers. We propose a control design based on the bang-bang funnel controller which is {robust with respect to uncertainties in the parameters} of the epidemiological model and only requires measurements of the number of individuals who require medical attention. Therefore, it may serve as a first step towards a reliable decision making mechanism. Simulations illustrate the efficiency of the proposed controller.
\end{abstract}

\end{frontmatter}


\section{Introduction}

One of the most difficult problems in forecasting the effects of the COVID-19 pandemic (or any other epidemic) is to generate reliable signals for {policy makers} in the presence of uncertain data and model parameters. To resolve this, in the present work a {robust} control approach is taken, which is able to generate reliable signals for social distancing measures (or against them) {in the presence of uncertainties in the parameters of the epidemiological model and which requires only measurements of the number} of COVID-19 patients showing moderate to severe symptoms (and independent of possibly asymptomatic infected or patients with mild symptoms). The number of those symptomatic infected, who require medical attention, can typically be measured accurately~-- {under the assumption that sufficient testing capabilities are available}.

It is well documented~\cite{KissTedi20,MaieBroc20} that social distancing measures help to reduce infection rates and have an effect on the containment of the spread of SARS-CoV-2. On the other hand, social distancing  has negative effects on both the economy and the mental and emotional health of the population. Therefore, {policy makers} face the hard decision of when to enact social distancing and when to relax the measures. The present paper may serve to provide a decision making mechanism based on a {robust} feedback control design.

{In contrast to model-based techniques such as MPC (model-predictive control) used in~\cite{GrunHeyd20,KohlSchw20,MoraBast20} for instance, in the present paper we use a control methodology which does not require a specific model and is robust with respect to uncertainties in the system parameters. Because of the latter, it is not necessary to precisely identify all parameters such as the infection, recovery or death rates. Therefore, the approach may allow for an easier scalability, i.e., it may be applicable to different countries or regions or cities, without the need to precisely (re-)identify all the parameters for this region.}

The above described question of when to enact social distancing measures or even a lockdown is a typical control theoretic question. Modelling this question utilizing a {control input which takes only a finite number of different values, as suggested e.g.\ in~\cite{KohlSchw20},} a suitable feedback controller is able to generate the required signals. In the present work we {restrict ourselves to} a binary control input, i.e., with values in $\{0,1\}$, {in order to first show the feasibility of the control design for this simple case.}  To this end, we combine a widely used model for the description of the COVID-19 pandemic {from~\cite{PiguShi20}} with a control component proposed in~\cite{MoraBast20}. The latter adds additional dynamics to the model which account for the effects of social distancing policies represented by the value of the control input and the response of the population to them (paying heed to possible delays).

The control objective is to keep the number of infected with moderate to severe symptoms (which typically require hospitalization) below a threshold defined by the number of available ICU (intensive care unit) beds. {Another approach discussing control strategies which are able to bound the hospitalized population, by using control barrier functions, can be found in~\cite{AmesMoln20}. In the present paper, to achieve the control objective} we exploit the bang-bang funnel controller developed in~\cite{LibeTren13b}, which is able to guarantee error margins in tracking problems and the control input switches between only two values. Funnel control proved an appropriate tool in several applications such as temperature control of chemical reactor models~\cite{IlchTren04}, termination of fibrillation processes~\cite{BergBrei21}, control of industrial servo-systems~\cite{Hack17} and underactuated multibody systems~\cite{BergDrue21,BergOtto19}, voltage and current control of electrical circuits~\cite{BergReis14a}, DC-link power flow control~\cite{SenfPaug14} and adaptive cruise control~\cite{BergRaue20}.

{We stress that the focus of the present paper is not on modelling aspects, for which we rely on the available literature. The essential contribution is to show the fundamental functionality of the proposed control law and to prove that it achieves that the available ICU capacity is not exceeded and that it is  robust
with respect to uncertainties in the parameters of the epidemiological model. Furthermore, we provide a lower bound for the interval of time between successive switching of the control input, which can be influenced by the design parameters.}

\section{Epidemiological model for the COVID-19 pandemic {and system class}}\label{Sec:Model}

\subsection{{Dynamics of the pandemic}}

{We use a so called SIRASD (Susceptible-Infected-Recovered-Asymptomatic-Symptomatic-Deceased) model for the dynamics of the COVID-19 pandemic from~\cite{MoraBast20}, which is able} to account for possible social distancing policies. {Note that recruitment is neglected in this model.} The resulting epidemiological model has the following dynamics:
\begin{equation}\label{eq:SIRASD}
\begin{aligned}
  \dot S(t) &= - \big(\beta_A \psi(t) I_A(t) + \beta_S \psi(t) I_S(t) \big) \tfrac{S(t)}{N-D(t)},\\
  \dot I_A(t) &= (1-p) \big(\beta_A \psi(t) I_A(t) + \beta_S \psi(t) I_S(t) \big) \tfrac{S(t)}{N-D(t)} \\
  &\quad - \alpha_A I_A(t),\\
  \dot I_S(t) &= p \big(\beta_A \psi(t) I_A(t) + \beta_S \psi(t) I_S(t) \big) \tfrac{S(t)}{N-D(t)}- \tfrac{\alpha_S}{1-\rho} I_S(t),\\
  \dot R(t) &= \alpha_A I_A(t) + \alpha_S I_S(t),\\
  \dot D(t) &= \tfrac{\rho \alpha_S}{1-\rho} I_S(t),\\
  \dot \psi(t) &= \gamma_0 (1-\psi(t)) (1-u(t)) + \gamma_1 \big( K_\psi(t) \bar \psi - \psi(t)\big) u(t),
\end{aligned}
\end{equation}
with the gain function
\begin{equation}\label{eq:Kpsi}
    K_\psi(t) = 1 - \gamma_K \tfrac{\rho\alpha_A}{1-\rho} \tfrac{I_A(t)}{N-D(t)}.
\end{equation}
In~\eqref{eq:SIRASD} the total population of a considered region is split into the following {compartments}:
\begin{itemize}
  \item susceptible individuals $S(t)$,\\[-6mm]
  \item infected but asymptomatic individuals $I_A(t)$,\\[-6mm]
  \item infected and symptomatic individuals $I_S(t)$,\\[-6mm]
  \item recovered individuals $R(t)$,\\[-6mm]
  \item deceased individuals $D(t)$ (due to the disease).
\end{itemize}
It is easily seen that the derivative of the sum of the above quantities is zero, $\ddt \big( S(t) + I_A(t) + I_S(t) + R(t) + D(t) \big) = 0$
for all $t\ge 0$, thus it stays constant over time and we may define the initial population (assuming $D(0)=0$) by
\begin{align*}
    N&:= S(0) + I_A(0) + I_S(0) + R(0)\\
    & = S(t) + I_A(t) + I_S(t) + R(t) + D(t),\quad t\ge 0.
\end{align*}
The other parameters used in~\eqref{eq:SIRASD} are summarized in Table~\ref{Tab:param}, where we have $\alpha_A, \alpha_S, \beta_A, \beta_S, \rho, p,\gamma_0, \gamma_1, \bar \psi \in [0,1]$.

\newcolumntype{g}{>{\columncolor{Gray}}c}
\begin{table}
\centering
\begin{tabular}{|c|p{6.3cm}|}
   \hline
   \rowcolor{Gray}
   Parameter & Epidemiological meaning \\
  \hline
  $\beta_A$, $\beta_S$ & transmission coefficients (contact or infection rates) for an asymptomatic
   or symptomatic individual to transmit the disease to a susceptible
     individual, resp. \\
  \rowcolor{Gray}
  $\alpha_A$, $\alpha_S$ & recovery rates for asymptomatic and symptomatic infected, resp.\\
  $p$ & proportion of individuals who develop symptoms \\
  \rowcolor{Gray}
  $\rho$ & probability of a symptomatic infected individual to die from the disease before recovering \\
  $\gamma_0, \gamma_1$ & settling-time parameters used to determine the average time of the population response\\
  \rowcolor{Gray}
  $\bar \psi$ & parameter which determines the strictest possible isolation \\
  $\gamma_K$ & positive gain coefficient\\
  \hline
\end{tabular}
\caption{Parameters of the SIRASD model~\eqref{eq:SIRASD}.}\label{Tab:param}
\end{table}

\subsection{{Population response}}

The last equation in~\eqref{eq:SIRASD} models the dynamics of social distancing policies and contains additional parameters to be explained in due course. {The simplest way to model the population response would be to replace the transmission coefficient $\beta_x$, where $x$ stands for either $A$ or $S$, by $\beta_x (1-u(t))$ as in~\cite{AmesMoln20}, i.e., the control directly influences the infection rates. However, the population response is typically not instantaneous, but people change their behavior with a certain delay~-- this is accounted for by the last equation in~\eqref{eq:SIRASD}.} As introduced in~\cite{MoraBast20}, the function~$\psi$ can be seen as a time-varying population response which decreases the transmission coefficients $\beta_A$ and $\beta_S$ in the case that social distancing measures are in place. Possible delays in the response are modelled by the parameters $\gamma_0, \gamma_1$. Note that we have $\psi(t) \in [0,1]$ for all $t\ge 0$, where $\psi(t) = 1$ stands for the case of no distancing at all and $\psi(t) = 0$ would mean that any contacts between people are suppressed. Of course, the latter case is unachievable in practice, which is accounted for in the model. The control input~$u$ is assumed to take only binary values, i.e., $u(t) \in \{0,1\}$ for all $t\ge 0$. This control signal models the policy enacted by the government, where $u(t) = 0$ means that no isolation measures are in place, and $u(t)=1$ means that policy makers have determined social distancing.

The dynamics with which these measures influence the response of the population are modelled by the last equation in~\eqref{eq:SIRASD}. Here, the value of the parameter $\bar \psi\in (0,1)$, see Table~\ref{Tab:param}, is additionally influenced by a gain function $K_\psi$ defined in~\eqref{eq:Kpsi}; the value of $K_\psi(t)$ decreases as the proportion of asymptomatic infected individuals increases. Assuming that initially $\psi(0) = 1$ (no isolation) we may infer that
\begin{equation}\label{eq:bound-psi}
    \forall\, t\ge 0:\ \psi(t) \in \left( \bar K_\psi \bar \psi, 1\right],
\end{equation}
where $\bar K_\psi = 1 - \gamma_K \frac{\rho\alpha_A}{1-\rho}$. For further details on the model~\eqref{eq:SIRASD} we refer to~\cite{MoraBast20,PiguShi20}.

\subsection{{Control objective}}\label{Ssec:ContrObj}

First note that, as in~\cite{MoraBast20}, we assume that the class of asymptomatic infected also includes those with only mild symptoms, which typically do not seek medical attention and are hence often not registered as infected. The class of symptomatic infected includes those with moderate to severe symptoms which may require ICU beds and are registered by local authorities -- thus, {provided that sufficient testing capabilities are available}, $I_S(t)$ is typically accurately measured.

The objective is to determine a control input signal $u: \R_{\ge 0} \to \{0,1\}$ which guarantees that the number of available ICU beds is not exceeded. This control signal may serve as an orientation for {policy makers} whether and when to enact social distancing measures. Since the class of symptomatic infected individuals encompasses those which may require ICU beds, we seek to keep $I_S(t)$ below the number $n_{\rm ICU}$ of available ICU beds, with tolerance {$\xi \ge 0$} accounting for symptomatic infected which do not require intensive care. In other words, the aim is to achieve
\begin{equation}\label{eq:ICU}
    \forall\, t\ge 0:\ I_S(t) < (1+\xi) n_{\rm ICU}.
\end{equation}

\subsection{{System class}}

{
We now turn to the specification of a class of systems described by the epidemiological model~\eqref{eq:SIRASD} amenable to the control strategy proposed in Section~\ref{Sec:Obj}. The class of systems of the form~\eqref{eq:SIRASD} is defined by the set of admissible parameters
\[
    \Sigma := \setdef{\!\!\!\begin{array}{l} (\alpha_A, \alpha_S, \beta_A, \beta_S, \rho, p, \\ \ \gamma_0, \gamma_1, \bar \psi, \gamma_K, \xi, n_{\rm ICU}, \\ \ S^0, I_A^0, I_S^0, R^0, D^0, \psi^0)\\ \ \in [0,1]^9 \times [0,\infty)^9\end{array}\!\!\!}{ \text{(A1)--(A3) hold}}.
\]
{Before stating the assumptions (A1)--(A3), we introduce the following definitions for a set of parameters from $\Sigma$:}
\begin{equation}\label{eq:constants}
    \left.
    \begin{aligned}
        &\vp^+ := (1+\xi) n_{\rm ICU},\\
        &N := S^0 + I_A^0 + I_S^0 + R^0,\\
        & S_{\min} := S^0 e^{-\tfrac{{\max\{\beta_A,\beta_S\}} (N-R^0)}{{\min\{\alpha_A,\alpha_S\}} R^0}},\\
        &{\tilde \beta := p\beta_S + (1-p)\beta_A,}\\
        &{A:= (1-p)\beta_A - p\beta_S + \tfrac{\left(\tfrac{\alpha_S}{1-\rho}-\alpha_A\right)N}{\bar K_\psi \bar\psi S_{\min}},}\\
        &{B := -\tfrac{A}{2} + \sqrt{\tfrac{A^2}{4}+p (1-p) \beta_A\beta_S},}\\
        &{\zeta := \max\left\{\tfrac{I_A^0}{I_S^0}, \tfrac{(1-p)\beta_S}{B}\right\},}\\
        &{M_1:=\bar K_\psi \bar\psi \tilde \beta \big(1-\tfrac{R^0}{N}\big) - \alpha_A},\\
        &{M_2:= (1+\bar K_\psi \bar\psi) \tfrac{\tilde\beta}{pN} - \tfrac{\rho\alpha_S}{(1-\rho) N}},\\
        &{M_3:= p(\beta_A  \zeta + \beta_S ) \big(1-\tfrac{R^0}{N}-\tfrac{M_2}{pN M_1}\big) \tfrac{(1-\rho) M_2}{\alpha_S M_1}}.
    \end{aligned}
    \ \ \right\}
    \end{equation}
{With these quantities the assumptions are}
\begin{enumerate}[\hspace{2pt}(A1)]
\item[(A1)] 
{$p>0$, $\rho<1$,  $0< \alpha_A \le \frac{\alpha_S}{1-\rho}$, 
$\gamma_K < \frac{1-\rho}{\rho \alpha_A}$, $M_1>0$} 
\item[(A2)] $S^0>0$, $R^0>0$, {$I_S^0>0$} and $I_A^0\ge \frac{1-p}{p} I_S^0$
\item[(A3)] 
    {$\vp^+ > \max\left\{ \tfrac{M_2}{M_1}, M_3\right\}$}

\end{enumerate}
}%
{The assumptions~(A1)--(A3) are mainly of technical nature and required in the proof of the main theorem. Nevertheless,} we like to note that~(A1)--(A3) are typically satisfied in real epidemiological scenarios, see also Section~\ref{Sec:Sim}. The initial values {$S^0, R^0, I_S^0$ in~(A2) for $S(0) = S^0, R(0)=R^0$ and $I_S(0)=I_S^0$} are assumed to be positive to avoid technicalities and keep the proofs simple. In fact, these assumptions are true for the COVID-19 pandemic in practically every country in the world and nearly every region by the date this article is written (and taken as $t=0$). Furthermore, in~(A2) it is assumed that {the initial value $I_A(0) = I_A^0$ satisfies} $I_A(0) \ge \frac{1-p}{p} I_S(0)$, which is not a restrictive assumption since the number of asymptomatic infected is typically much larger than the number of symptomatic infected individuals {and usually} the case at the beginning of a (local) outbreak. {Finally, assumption~(A3) {defines a lower bound for the available ICU capacity $\vp^+$ and} is required to determine a non-empty range of control design parameters, see Section~\ref{Sec:Obj}.} {It is easy to see that the assumptions are not contradictory and hence $\Sigma$ is non-empty.}

\section{Controller design}\label{Sec:Obj}

In order to obtain a feedback control law which is able to guarantee the transient behavior {in~\eqref{eq:ICU} (and hence to achieve the control objective)}, we exploit the idea of bang-bang funnel control from~\cite{LibeTren13b}. ``Bang-bang'' means that the control input switches between only two different values, hence this technique is suitable for our purposes. We stress that system~\eqref{eq:SIRASD} does not belong to the class of systems investigated in~\cite{LibeTren13b} and hence feasibility of bang-bang funnel control needs to be investigated separately. However, due to the special structure of~\eqref{eq:SIRASD} a simpler control law is possible here. To be precise, we consider the controller
\begin{equation}\label{eq:BBFC}
  u(t) = \begin{cases}
    1, & \text{if}\ \ I_S(t) \ge \varphi^+ - \eps^+,\\
    0, & \text{if}\ \  I_S(t) \le \vp^- + \eps^-,\\
    u(t-), & \text{otherwise,}
  \end{cases}
\end{equation}
where $\vp^+, \vp^-, \eps^+, \eps^-$ are non-negative parameters satisfying  $\vp^- + \eps^- < \varphi^+ - \eps^+$ and by $u(t-)$ we denote the left limit $u(t-) = \lim_{h\searrow 0} u(t-h)$ of the piecewise constant function~$u$ at~$t$. The controller is initialized by $u(0-) = 0$.

In~\cite{LibeTren13b}, $\vp^+$ and $\vp^-$ are time-varying functions which determine a performance funnel for a certain signal to evolve in and $\eps^+, \eps^-$ are ``safety distances'' required to guarantee that the signal evolves within this funnel. In the present paper, the aim is to achieve $I_S(t) \in [\vp^- , \vp^+]$, thus, in view of the control objective {formulated in Subsection~\ref{Ssec:ContrObj}}, we may choose
\begin{equation}\label{eq:vp+-}
    \vp^- := 0,\quad \vp^+ := (1+\xi) n_{\rm ICU}.
\end{equation}
We emphasize that the feedback control strategy~\eqref{eq:BBFC} only requires the measurement of the number $I_S(t)$ of symptomatic infected individuals at time~$t$, which, as mentioned in Section~\ref{Sec:Model}, can be measured accurately.

{In the following we identify conditions on the ``safety distances'' $\eps^+, \eps^-$ so that the application of the controller~\eqref{eq:BBFC} to a system~\eqref{eq:SIRASD} with a tuple of parameters from $\Sigma$ is feasible.} In essence, these assumptions mean that by maintaining a strict lockdown, i.e., $u(t) = 1$ for $t\ge 0$, it is possible to guarantee~\eqref{eq:ICU}. If this is not possible, then no switching strategy can be successful.

{The pairs of feasible control parameters depend on the choice of system parameters in general. For any $Z\in\Sigma$ we define the control parameter set
\[
    C_Z := \setdef{(\eps^-,\eps^+)\in (0,\infty)^2 }{\!\!\begin{array}{l} \text{{$\eps^-<\vp^+-\eps^+$ and}}\\ \text{(A4)--(A5) hold}\end{array}\!\!\!},
\]
under the assumptions, where we use the constants defined in~\eqref{eq:constants} in terms of~$Z$,
\begin{enumerate}[\hspace{2pt}(A1)]
\item[(A4)] 
{$\eps^+  < \vp^+ - \tfrac{M_2}{M_1} $},\\[-2mm]
\item[(A5)] {$\frac{p(\beta_A  \zeta + \beta_S ) \eps \big(1-\tfrac{R^0}{N}-\tfrac{\eps}{pN}\big) + p \big(M_1\eps  - M_2\big) (\zeta+1)\eps}{\tfrac{\alpha_S}{1-\rho} +  \big(M_1\eps  - M_2\big)} < \vp^+$ for $\eps = \vp^+-\eps^+$}.\\[-2mm]
\end{enumerate}
}
Assumptions~(A4) and~(A5) are quite conservative since they are designed for worst-case scenarios, and hence they may impose hard restrictions in practice. However, the controller~\eqref{eq:BBFC} may even be feasible if these assumptions are not satisfied and appropriate values for $\eps^+$ and $\eps^-$ can be identified by simulations.

{It can be seen that, while~(A4) simply gives a lower bound for $\eps^-$, assumption~(A5) requires $\eps^+$ to be within a range, where both the lower and upper bound depend on $\eps^-$. Therefore, the question arises whether~(A4) and~(A5) can be simultaneously satisfied. The following result gives an affirmative answer to this question, which is based on the assumptions~(A1)--(A3) on the system parameters~$Z$.}

\begin{lemma}\label{lem:CZnonempty}
  For all $Z\in\Sigma$ the set $C_Z$ is non-empty and open.
\end{lemma}
\begin{proof} We show that there exist  $\eps^-, \eps^+> 0$ which satisfy~(A4) and~(A5). This is true if, and only if, there exists {$\frac{M_2}{M_1}<\eps<\vp^+$} which satisfies
\begin{equation}\label{eq:proof-CZ}
    {q(\eps) := \tfrac{p(\beta_A  \zeta + \beta_S ) \eps \big(1-\tfrac{R^0}{N}-\tfrac{\eps}{pN}\big) + p \big(M_1\eps  - M_2\big) (\zeta+1)\eps}{\tfrac{\alpha_S}{1-\rho} +  \big(M_1\eps  - M_2\big)} < \vp^+.}
\end{equation}
{Then we may define $\eps^+ := \vp^+-\eps$ and choose $\eps^-<\eps$, which satisfy $(\eps^-,\eps^+)\in C_Z$. The existence of $\eps$ as required above follows immediately from the observation that $q(\cdot)$ is continuous on the interval $\left[\tfrac{M_2}{M_1},\vp^+\right]$ and $q\left(\tfrac{M_2}{M_1}\right) = M_3 < \vp^+$ by~(A3).}
%
This shows that the set $C_Z$ is non-empty. Finally, from assumptions (A4) and (A5) it is clear that $C_Z$ is open.
\end{proof}

{Since the values of the input~$u$ have the interpretation of enacted social distancing policies, it is typically not desired that fast switching between the two binary values occurs. To this end, in the remainder of this section we derive a lower bound for the interval of time between successive switches, the so-called dwell time. If $u(t_0)=1$ with $u(t_0-)=0$, then $I_S(t_0) = \vp^+-\eps^+$ and as long as $I_S(t)>\eps^-$ we have that $u(t)=1$. Thus the dwell time is $t_1-t_0$ with $t_1 = \inf\setdef{t\ge t_0}{I_S(t) = \eps^-}$. Similarly, if $u(t_0)=0$ with $u(t_0-)=1$, then $I_S(t_0) = \eps^-$ and as long as $I_S(t)<\vp^+-\eps^+$ we have that $u(t)=0$. Then the dwell time is $t_1-t_0$ with $t_1 = \inf\setdef{t\ge t_0}{I_S(t) = \vp^+-\eps^+}$. Lower bounds for the dwell times are given in the following result.}

\begin{proposition}\label{Prop:dwell_time}
{Let $Z\in\Sigma$, $(\eps^-,\eps^+)\in C_Z$ and assume that $(S,I_A,I_S,R,D,\psi):\R_{\ge 0}\to \R^6$ is a global solution of~\eqref{eq:SIRASD} under the control~\eqref{eq:BBFC}. Then we have the following implications for all $0\le t_0\le t_1$:
\begin{align*}
  I_S(t_0) &= \vp^+-\eps^+\ \wedge\ I_S(t_1) = \eps^-\\
   &\implies\quad t_1-t_0 \ge \frac{1-\rho}{\alpha_S} \ln \left(\frac{\vp^+-\eps^+}{\eps^-}\right) > 0,\\
  I_S(t_0) &= \eps^-\ \wedge\ I_S(t_1) =\vp^+-\eps^+ \\
  &\implies\quad t_1-t_0 \ge\frac{1}{\mu} \ln \left(\frac{(\vp^+-\eps^+)^2}{(\eps^-)^2 + I_A(t_0)^2}\right),
\end{align*}
where $\mu = \max\big\{\tfrac{1+p}{2}\beta_S + \tfrac{p}{2}\beta_A - \frac{\alpha_S}{1-\rho}, \frac{2-p}{2}\beta_A +\frac{1-p}{2}\beta_S-\alpha_A,\delta\big\}$ for some arbitrary $\delta>0$.}
\end{proposition}
\begin{proof}{
First consider the case that $I_S(t_0) = \vp^+-\eps^+$ and $I_S(t_1) = \eps^-$. From~\eqref{eq:SIRASD},~\eqref{eq:bound-psi} and Lemma~\ref{lem:pos_sln} we see immediately that $\dot I_S(t) \ge -\frac{\alpha_S}{1-\rho} I_S(t)$, thus
\[
    \eps^- = I_S(t_1) \ge e^{-\frac{\alpha_S}{1-\rho}(t_1-t_0)} I_S(t_0) = e^{-\frac{\alpha_S}{1-\rho}(t_1-t_0)} \big(\vp^+-\eps^+),
\]
from which the first implication follows. Note that $\tfrac{\vp^+-\eps^+}{\eps^-}>1$ and hence the lower bound is positive.\\
Next consider the case that $I_S(t_0) = \eps^-$ and $I_S(t_1) = \vp^+-\eps^+$. From~\eqref{eq:SIRASD},~\eqref{eq:bound-psi}, Lemma~\ref{lem:pos_sln} and the fact that $S(t)\le N-D(t)$ we may infer that
\begin{align*}
  \dot I_A(t) &\le \big((1-p)\beta_A - \alpha_A\big) I_A(t) + (1-p)\beta_S I_S(t),\\
  \dot I_S(t) &\le \big(p\beta_S-\tfrac{\alpha_S}{1-\rho}\big) I_S(t) + p\beta_A I_A(t).
\end{align*}
Define $z(t) = \tfrac12\big( I_A(t)^2 + I_S(t)^2\big)$, then
\begin{align*}
  \dot z(t) &=  \big((1-p)\beta_A - \alpha_A\big) I_A(t)^2 + (1-p)\beta_S I_S(t) I_A(t) \\
  &\quad + \big(p\beta_S-\tfrac{\alpha_S}{1-\rho}\big) I_S(t)^2 + p\beta_A  I_A(t)I_S(t) \\
  &\le \big((1-p)\beta_A - \alpha_A + \tfrac{1-p}{2}\beta_S + \tfrac{p}{2} \beta_A\big) I_A(t)^2 \\
  &\quad +\big(p\beta_S-\tfrac{\alpha_S}{1-\rho} + \tfrac{1-p}{2}\beta_S + \tfrac{p}{2} \beta_A\big) I_S(t)^2\\
  &\le \mu z(t)
\end{align*}
for all $t\ge 0$, thus
\begin{align*}
  (\vp^+-\eps^+)^2 &= I_S(t_1)^2 \le 2 z(t_1) \le 2 e^{\mu (t_1-t_0)} z(t_0)\\
  &=e^{\mu (t_1-t_0)} \big(I_A(t_0)^2 + (\eps^-)^2\big),
\end{align*}
from which the second implication follows.
}
\end{proof}

{Note that the lower bound for the dwell time in $u(t)=0$ in Proposition~\ref{Prop:dwell_time} depends on $I_A(t_0)$, and if this value is very large, then the lower bound may become negative, thus not giving any result. Furthermore, both lower bounds essentially depend on the parameters $\vp^+,\eps^+,\eps^-$, which can be adjusted in order to shape the minimum dwell time. In particular, for $\vp^+\to\infty$ any desired minimum dwell time can be achieved (since $I_A(t_0)$ is bounded by~$N$).}

\section{Main results}

\subsection{{Feasibility of the control design}}

In the following {first} main result of the present paper we prove that, {for any tuple of system parameters $Z\in\Sigma$ and any pair of controller parameters $(\eps^-,\eps^+)\in C_Z$}, the application of the bang-bang control law~\eqref{eq:BBFC} {with parameters $(\eps^-,\eps^+)$} to the epidemiological model~\eqref{eq:SIRASD} {with parameters $Z$ leads to a closed-loop system, which has a global and bounded solution that satisfies~\eqref{eq:ICU}.} Moreover, the control input~$u$ has only a finite number of jumps in each compact set. Solutions are considered in the sense of Carath\'{e}odory, i.e., they are assumed to be locally absolutely continuous and satisfy the differential equation almost everywhere. A solution is said to be maximal, if it has no right extension that is also a solution.

\begin{theorem}\label{Thm:BBFC}
{Let $Z\in\Sigma$ and consider the associated system~\eqref{eq:SIRASD}. Further let $(\eps^-,\eps^+)\in C_Z$ be a pair of controller parameters, $\vp^+$ be as in~\eqref{eq:vp+-} and} assume that
\[
    I_S^0 \in [0, \vp^+-\eps^+],\quad D^0 = 0\quad \text{and}\quad \psi^0 = 1.
\]
Then the controller~\eqref{eq:BBFC} applied to~\eqref{eq:SIRASD}, {under the initial conditions $S(0)=S^0, I_A(0)=I_A^0, I_S(0)=I_S^0, R(0)=R^0, D(0)=D^0, \psi(0)=\psi^0$}, leads to a closed-loop system which has a global and bounded solution $(S,I_A,I_S,R,D,\psi) : \R_{\ge 0}\to (\R_{\ge 0})^6$ such that~\eqref{eq:ICU} holds and~$u$ defined by~\eqref{eq:BBFC} has {finitely many jumps in each compact set}.
\end{theorem}
\begin{proof}
{Throughout the proof we use the constants defined in~\eqref{eq:constants} without further notice.} We divide the proof into several steps, {where we first show the existence of a local solution (Step~1). In Step~2 we show that~$u$ has only finitely many jumps in each compact set, which enables us to prove that the solution is actually global (Step~3). It then remains to show~\eqref{eq:ICU} in Step~4.}

\emph{Step 1}: We show the existence of a local solution. In fact, repeating the arguments of~\cite[Cor.~3.3]{LibeTren10pp} or~\cite[Thm.~5.3]{LibeTren13b} yields a maximal solution  $(S,I_A,I_S,R,D,\psi) : [0,\omega) \to \R^6$ of~\eqref{eq:SIRASD},~\eqref{eq:BBFC} with $\omega \in (0,\infty]$. {Clearly, $S,I_A,I_S,R,D$ are continuously differentiable, since they are absolutely continuous and by~\eqref{eq:SIRASD} their derivatives are continuous.} Note that clearly the solution is bounded as each component is non-negative {(cf.\ Lemma~\ref{lem:pos_sln}\,(i))} and we have $S(t) + I_A(t) + I_S(t) + R(t) + D(t) = N$ and $\psi(t) \le 1$ for all $t\in [0,\omega)$. Therefore, we have that $N - D(t) = S(t) + I_A(t) + I_S(t) + R(t) \ge S(t)$ by which $\frac{S(t)}{N-D(t)} \le 1$ for all $t\in [0,\omega)$ and hence the respective quotient in~\eqref{eq:SIRASD} is uniformly bounded.


\emph{Step 2}: We show that~$u$ has only finitely many jumps in each interval $[a,b)\subseteq [0,\omega)$. {This is an immediate consequence of the first implication in Proposition~\ref{Prop:dwell_time}.} 

\emph{Step 3}: We show that $\omega = \infty$. Since $(S,I_A,I_S,R,D,\psi)$ is bounded as $S(t) + I_A(t) + I_S(t) + R(t) + D(t) = N$ and $\psi(t) \le 1$ for all $t\in [0,\omega)$, the case $\omega < \infty$ is only possible when the jumps in~$u$ accumulate for $t\to\omega$, but this is excluded by Step~2.

\emph{Step 4}: It remains to show~\eqref{eq:ICU}. {For brevity, set $\eps := \vp^+ - \eps^+$. Let $t_1 := \inf \setdef{t\ge 0}{ I_S(t) = \vp^+}$ and, seeking a contradiction, assume that $t_1 < \infty$. Set $t_0:=\max \setdef{t\in[0,t_1)}{I_S(t) = \eps}$, then $I_S(t_0)=\eps$, $I_S(t)\le \vp^+$ for all $t\in[0,t_1]$ and $I_S(t)\ge \eps$ for all $t\in[t_0,t_1]$. We show that $I_S(t_1) < \vp^+$, which contradicts the assumptions and thus proves the claim.}

{Note that $\dot I_S$ is almost everywhere differentiable since $u$ is piecewise constant and hence $\psi$ is almost everywhere differentiable.} 
Furthermore, by~\eqref{eq:BBFC} we have $u(t) = 1$ for all $t\in [t_0,t_1]$ and hence $\dot \psi(t) = \gamma_1 \big( K_\psi(t) \bar \psi - \psi(t)\big)<0$. Then we may calculate
\begin{align*}
  &\ddot I_S(t) = p \dot \psi(t) \big( \beta_A I_A(t) + \beta_S I_S(t)\big) \tfrac{S(t)}{N-D(t)}- \tfrac{\alpha_S}{1-\rho} \dot I_S(t)\\
   &\quad + p \psi(t) \big( \beta_A \dot I_A(t) + \beta_S \dot I_S(t)\big) \tfrac{S(t)}{N-D(t)}\\
  &\quad + p \psi(t) \big( \beta_A I_A(t) + \beta_S I_S(t)\big) \tfrac{\dot S(t) (N-D(t)) + S(t) \dot D(t)}{(N-D(t))^2} \\
  &{\stackrel{\eqref{eq:SIRASD}}{\le}  - \tfrac{\alpha_S}{1-\rho} \dot I_S(t) + \tfrac{p \psi(t) S(t)}{N-D(t)} \Big( \!-\!\beta_A \big((1\!-\!p) \dot S(t) \!+\! \alpha_A I_A(t)\big)} \\
  &\quad {- \beta_S \big(p\dot S(t) +  \tfrac{\alpha_S}{1-\rho} I_S(t)\big) + \big( \beta_A I_A(t) + \beta_S I_S(t)\big) \big(\tfrac{\dot S(t)}{S(t)}} \\
  &\quad {  + \tfrac{\rho\alpha_S}{1-\rho} \tfrac{I_S(t)}{N-D(t)}\big)\Big)} \\
  &{\stackrel{\rm (A1),\eqref{eq:SIRASD}}{\le}  - \tfrac{\alpha_S}{1-\rho} \dot I_S(t) - \tfrac{p \psi(t) S(t)}{N-D(t)} \Big( \tilde \beta \dot S(t) + \alpha_A \big(\beta_A I_A(t) } \\
  &\quad {+ \beta_S I_S(t)\big)\Big) +p \dot S(t) \Big(\big( \beta_A I_A(t) + \beta_S I_S(t)\big) \tfrac{\psi(t)}{N-D(t)}} \\
  &\quad {  - \tfrac{\rho\alpha_S}{1-\rho} \tfrac{I_S(t)}{N-D(t)}\big)\Big)} \\
  &{\stackrel{\eqref{eq:SIRASD}}{=}  - \tfrac{\alpha_S}{1-\rho} \dot I_S(t) + p \dot S(t)  \Big( -\tilde \beta \tfrac{\psi(t) S(t)}{N-D(t)} + \alpha_A }\\
  &\quad {+ \big( \beta_A I_A(t) + \beta_S I_S(t)\big) \tfrac{\psi(t)}{N-D(t)}  - \tfrac{\rho\alpha_S}{1-\rho} \tfrac{I_S(t)}{N-D(t)}\Big)} \\
  &{\stackrel{\eqref{eq:bound-psi},\rm Lem.\!\,\ref{lem:pos_sln}\,(ii)}{\le}  - \tfrac{\alpha_S}{1-\rho} \dot I_S(t) + p \dot S(t)  \Big( \alpha_A - \tilde \beta \bar K_\psi \bar \psi \tfrac{S(t)}{N-D(t)}}\\
  &\quad {+ \big( \tfrac{\tilde\beta}{p}  - \tfrac{\rho\alpha_S}{1-\rho} \big)\tfrac{I_S(t)}{N-D(t)}\Big)} \\
  &{\stackrel{\rm Lem.\!\,\ref{lem:pos_sln}\,(ii)}{\le}  - \tfrac{\alpha_S}{1-\rho} \dot I_S(t) + p \dot S(t)  \Big( \alpha_A -  \tfrac{\tilde \beta \bar K_\psi \bar \psi}{N-D(t)} \big(N-D(t)}\\
  &\quad {-R(t)-\tfrac{1}{p} I_S(t)\big)+ \big( \tfrac{\tilde\beta}{p}  - \tfrac{\rho\alpha_S}{1-\rho} \big)\tfrac{I_S(t)}{N-D(t)}\Big)} \\
  &{\stackrel{\rm (A1)}{\le}  - \tfrac{\alpha_S}{1-\rho} \dot I_S(t) + p \dot S(t)  \Big( \alpha_A -  \tilde \beta \bar K_\psi \bar \psi \big(1-\tfrac{R_0}{N}\big)}\\
  &\quad {+ \big( (1+\bar K_\psi \bar \psi ) \tfrac{\tilde\beta}{pN}  - \tfrac{\rho\alpha_S}{(1-\rho)N} \big) I_S(t)\Big)} \\
  &{{\le}  - \tfrac{\alpha_S}{1-\rho} \dot I_S(t) + p \dot S(t) \big(M_1 \eps - M_2\big)}
\end{align*}
for {almost} all $t\in [t_0,t_1]$, {where we have used that $I_S(t) \ge \eps$ in the last inequality. By~(A4) we have that $ M_1\eps  - M_2 > 0$ for $M_1, M_2$ defined in~\eqref{eq:constants}.} 
{Upon integration we obtain that
\begin{align*}
  &\dot I_S(t) = \dot I_S(t_0) + \int_{t_0}^t \ddot I_S(s) {\rm d}s \\
  &\stackrel{\eqref{eq:bound-psi}}{<} p(\beta_A I_A(t_0) \!+\! \beta_S I_S(t_0))\left(1-\tfrac{R(t_0)+I_A(t_0)+I_S(t_0)}{N-D(t_0)}\right)  \\
  &\quad - \tfrac{\alpha_S}{1-\rho} I_S(t_0) - \tfrac{\alpha_S}{1-\rho} (I_S(t) - I_S(t_0)) \\
  &\quad + p (S(t)-S(t_0))  \big(M_1\eps  - M_2\big) \\
  &\stackrel{\rm Lem.\!\,\ref{lem:pos_sln}\,(ii),(iv)}{<}  p\left(\beta_A  \zeta + \beta_S \right) \eps \left(1-\tfrac{R^0}{N}-\tfrac{\eps}{pN}\right) - \tfrac{\alpha_S}{1-\rho} I_S(t)   \\
   &\quad  + p \big(M_1\eps  - M_2\big) \big(I_A(t_0) - I_A(t) + I_S(t_0) -  I_S(t)\big) \\
 &\stackrel{\rm Lem.\!\,\ref{lem:pos_sln}\,(ii),(iv)}{<}  p\left(\beta_A  \zeta + \beta_S \right) \eps \left(1-\tfrac{R^0}{N}-\tfrac{\eps}{pN}\right) - \tfrac{\alpha_S}{1-\rho} I_S(t)   \\
   &\quad  + p \big(M_1\eps  - M_2\big) \big((\zeta+1)\eps - \tfrac{1}{p} I_S(t)\big)\\
   &= M_3(\eps) - M_4(\eps) I_S(t)
\end{align*}
for all $t\in [t_0,t_1]$, where $M_3(\eps) :=  p\left(\beta_A  \zeta + \beta_S \right) \eps \left(1-\tfrac{R^0}{N}-\tfrac{\eps}{pN}\right) + p \big(M_1\eps  - M_2\big) (\zeta+1)\eps$ and $M_4(\eps) := \tfrac{\alpha_S}{1-\rho} +  \big(M_1\eps  - M_2\big)$. Then Gr\"onwall's lemma gives
\begin{align*}
  I_S(t) &< e^{-M_4(\eps) (t-t_0)} I_S(t_0) + \int_{t_0}^t e^{-M_4(\eps)(t-s)} M_3(\eps) {\rm d}s\\
  &= \left( \eps - \tfrac{M_3(\eps)}{M_4(\eps)}\right) e^{-M_4(\eps) (t-t_0)} + \tfrac{M_3(\eps)}{M_4(\eps)} =: q(t)
\end{align*}
for all  $t\in [t_0,t_1]$. If $\eps \le \tfrac{M_3(\eps)}{M_4(\eps)}$, then~(A5) gives that $q(t)\le \tfrac{M_3(\eps)}{M_4(\eps)} <\vp^+$ for all $t\in [t_0,t_1]$, and if $\eps > \tfrac{M_3(\eps)}{M_4(\eps)}$, then $q(\cdot)$ is monotonically decreasing on $[t_0,t_1]$, thus $q(t) \le q(t_0) = \eps < \vp^+$.}
This finishes the proof.
\end{proof}

We like to note that in Theorem~\ref{Thm:BBFC} it is assumed that the initial value $I_S(0)$ lies within a certain interval. If the upper bound $I_S(0) \le \vp^+-\eps^+$ does not hold, then even a strict lockdown with $u(t)=1$ for $t\ge 0$ may not be sufficient to guarantee~\eqref{eq:ICU}.

\subsection{{Robustness of the control design}}

In the following second main result we show that, when the controller parameters $(\eps^-,\eps^+)$ are fixed, then there exists an open subset of the set of admissible system parameters, so that the controller~\eqref{eq:BBFC} with $(\eps^-,\eps^+)$ is feasible for every system~\eqref{eq:SIRASD} with parameters from this open set. In other words, the controller~\eqref{eq:BBFC} is robust with respect to uncertainties in the system parameters. {In order to show this result we need to consider a subset $\Sigma_{\rm rob}$ of~$\Sigma$, where the parameters satisfy an additional assumption:
\[
    \Sigma_{\rm rob} := \setdef{Z\in\Sigma}{ \text{(A6) holds}},
\]
where
\begin{enumerate}[\hspace{2pt}(A1)]
\item[(A6)] $\left( \tfrac{1}{M_2} - \tfrac{1-\rho}{\alpha_S}\right) \left( pNM_1 - pR^0 M_1 - M_2\right) > 1$ and $p N M_1 (\zeta+1) > \beta_A \zeta + \beta_S$.\\[-2mm]
\end{enumerate}
Assumption~(A6) is of pure technical nature and typically satisfied in real epidemiological scenarios.}
Since the feasibility of~\eqref{eq:BBFC} follows from Theorem~\ref{Thm:BBFC}, it is sufficient to show that $(\eps^-,\eps^+)\in C_Z$ is contained in any set $C_{\tilde Z}$ for $\tilde Z$ in a neighborhood of~$Z\in \Sigma_{\rm rob}$.

\begin{theorem}\label{Thm:robust}
For all interior points $Z$ of {$\Sigma_{\rm rob}$} and all $(\eps^-,\eps^+)\in C_Z$ there exists $\delta>0$ such that $B_\delta(Z) \subseteq {\Sigma_{\rm rob}}$ and
\[
    \forall\, \tilde Z \in B_\delta(Z):\ (\eps^-,\eps^+)\in C_{\tilde Z},
\]
where $B_\delta(Z)$ denotes the open ball with radius~$\delta$ around~$Z$ in $[0,1]^9\times [0,\infty)^9$.
\end{theorem}
\begin{proof}
Let $Z\in{\Sigma_{\rm rob}}$ and $(\eps^-,\eps^+)\in C_Z$ be fixed. Since $Z$ is an interior point of ${\Sigma_{\rm rob}}$ there exists $\delta_1>0$ such that $B_{\delta_1}(Z)\subseteq {\Sigma_{\rm rob}}$. {First, consider the function $q:[\tfrac{M_2}{M_1},\vp^+]\to \R$ defined in~\eqref{eq:proof-CZ}, where the appearing constants are defined in terms of~$Z$. We show that $q$ is strictly monotonically increasing. To this end, observe that $q'(\eps) = \tfrac{q_1(\eps)}{q_2(\eps)^2}$ with $q_2(\eps) = \tfrac{\alpha_S}{1-\rho} +  M_1\eps  - M_2$ and
\begin{align*}
  q_1(\eps) &= \Big(p z \big( 1 - \tfrac{R^0}{N} - \tfrac{2\eps}{pN}\big) + p (\zeta+1) ( 2M_1 \eps - M_2)\Big) q_2(\eps)\\
  & - p z M_1 \eps \big( 1 - \tfrac{R^0}{N} - \tfrac{\eps}{pN}\big) - p M_1  (\zeta+1) \eps (M_1\eps - M_2),
\end{align*}
where $z := \beta_A \zeta + \beta_S$. Then $q_1(\tfrac{M_2}{M_1}) > 0$ if, and only if,
\begin{align*}
    &0< \tfrac{\alpha_S}{1-\rho}\big( 1 - \tfrac{R^0}{N} - \tfrac{2M_2}{pNM_1}\big) - M_2 \big( 1 - \tfrac{R^0}{N} - \tfrac{M_2}{pN M_1}\big) \\
    &\Longleftarrow\quad \text{(A6)}.
\end{align*}
Furthermore,
\[
    q_1'(\eps) = 2\big( p(\zeta+1) M_1 - \tfrac{z}{N}\big) q_2(\eps),
\]
which is positive if, and only if, $p(\zeta+1) M_1 - \tfrac{z}{N}> 0$ and this is guaranteed by the second condition in~(A6). We have now shown that $q_1$ is strictly increasing on $[\tfrac{M_2}{M_1},\vp^+]$ and positive at $\eps = \tfrac{M_2}{M_1}$, thus $q'(\eps)$ is positive on $[\tfrac{M_2}{M_1},\vp^+]$ and hence $q$ is strictly increasing on $[\tfrac{M_2}{M_1},\vp^+]$ and therefore invertible.}

{For any $\tilde Z\in\Sigma_{\rm rob}$, let $q_{\tilde Z}$ be the corresponding function as in~\eqref{eq:proof-CZ} as considered above. Define the functions
\begin{align*}
  &f_1:B_{\delta_1}(Z)\times\R \to \R,\ (\tilde Z,s)\mapsto \vp^+-s,\\
  &f_2:B_{\delta_1}(Z)\to\R,\ \tilde Z\mapsto \vp^+ - \tfrac{M_2}{M_1},\\
  &f_3:B_{\delta_1}(Z) \to \R,\ \tilde Z\mapsto q_{\tilde Z}^{-1}(\vp^+),
\end{align*}
then} we find that, for all $\tilde Z\in B_{\delta_1}(Z)$,
\[
    C_{\tilde Z} = \setdef{(s,t)\in (0,\infty)^2}{\!\!\begin{array}{l} s\in {(0,f_1(\tilde Z,t))},\\ t\in {(f_2(\tilde Z),f_3(\tilde Z))}\end{array}\!\!\!\!}.
\]
Now we have that
\[
    \eps^- \in {(0,f_1(Z,\eps^+))},\quad \eps^+ \in {(f_2(Z),f_3(Z)),}
\]
and since $f_3,f_4$ are continuous and {$\tilde Z \mapsto f_1(\tilde Z,\eps^+)$} is continuous as well on $B_{\delta_1}(Z)$, there exists $\delta\in (0,\delta_1)$ such that for all $\tilde Z \in B_\delta(Z)$ we have
\[
    {\eps^- \in (0,f_1(\tilde Z,\eps^+)),\quad \eps^+ \in (f_2(\tilde Z),f_3(\tilde Z)),}
\]
which finishes the proof.
\end{proof}

{
While we have shown in Theorem~\ref{Thm:robust} that, for a fixed pair of controller parameters, the control strategy~\eqref{eq:BBFC} is robust with respect to uncertainties in the system parameters, it still depends on the structure of the model~\eqref{eq:SIRASD}. Nevertheless, if a different epidemiological model is chosen, replacing~\eqref{eq:SIRASD}, then {Theorem~\ref{Thm:BBFC} and its proof can be accordingly adjusted (even though a new analysis would be necessary) so that} the controller~\eqref{eq:BBFC} is still feasible for a suitable choice of parameters $(\eps^-,\eps^+)$ and its robustness properties are retained.
}

\section{Simulations}\label{Sec:Sim}

In this section we illustrate our findings by a simulation of the epidemiological model~\eqref{eq:SIRASD} under the feedback control law~\eqref{eq:BBFC}. 
For the simulation we {consider a population of $N=10^5$ individuals in an \textit{Example City}.} At time $t=0$ the {population of \textit{Example City}} is divided into
\begin{align*}
    &{S(0) = 0.9 \cdot 10^5 - 50,\ I_A(0) = 49,\ I_S(0) = 1,}\\
    &{R(0) = 0.1\cdot 10^5,}\ D(0) = 0,
\end{align*}
i.e., we have {fifty symptomatic or} asymptomatic infected 
and we assume that {$10\%$} of the population already developed immunity due to a prior disease and hence belong to the class of recovered individuals. This number may be reasonable for the beginning of the COVID-19 pandemic (at least in some regions in Germany) as e.g.\ the COVID-19 Case-Cluster-Study~\cite{Stre20} (Heinsberg study) suggests.

{We further assume that the disease spreads in \textit{Example City} according to the} parameters that have been used in~\cite{MoraBast20} {and which are} summarized in Table~\ref{Tab:sim-param}. {Note that we chose a smaller value for~$p$ than in~\cite{MoraBast20}, since the number of patients who require intensive care is typically much lower than the number of symptomatic infected in the model from~\cite{MoraBast20}.}

\begin{table}[ht]
\centering
\resizebox{\columnwidth}{!}{
\begin{tabular}{c|c|c|c|c|c|c|c|c|c}
    $\beta_A$ & $\beta_S$ & $\alpha_A$ & $\alpha_S$ & $p$ & $\rho$ & $\gamma_0$ & $\gamma_1$ & $\bar \psi$ & $\gamma_K$ \\
    \hline
    0.37 & 0.43 & 0.1 & 0.085 & 0.02 & 0.15 & 1 & 1 & 0.31 & 1
\end{tabular}
}
\caption{Parameter values for the SIRASD model~\eqref{eq:SIRASD} taken from~\cite[Sec.~3.2]{MoraBast20} for the case of the ``Uncertain 1'' model. {Only the values of $\alpha_S,p$ and of $\gamma_0,\gamma_1,\bar \psi$ for the population response are different from~\cite{MoraBast20}.}}\label{Tab:sim-param}
\end{table}

The number of available ICU beds {in \textit{Example City} is assumed to correspond to the available} capacity in Germany, which are $32637$ beds (based on data from \url{de.statista.com}, July 21, 2020) for a population of about $8.3\cdot 10^7$. Preserving this ratio, we obtain for {\textit{Example City} with} $N$ inhabitants a number of $n_{\rm ICU} = 40$ (rounded to the next integer) ICU beds. As tolerance we choose $\xi = 0.1$ and we recall that $\vp^+ = (1+\xi) n_{\rm ICU}$ is the threshold that the number $I_S$ of symptomatic infected should not exceed, cf.~\eqref{eq:ICU}. For purposes of comparison, Fig.~\ref{fig:sim-u=0} shows the simulation results in the case that the {policy makers} take no action and hence $u(t) = 0$ for all $t\ge 0$. It can clearly be seen that the number $I_S(t)$ quickly increases above the available ICU capacity and the number of deaths rises dramatically; note that the effect of the exceeded ICU capacity is not considered in the model and hence the number of deaths might even be much higher.

\begin{figure}[h]
\begin{subfigure}{.9\linewidth}
\begin{center}
  \includegraphics[scale=0.6]{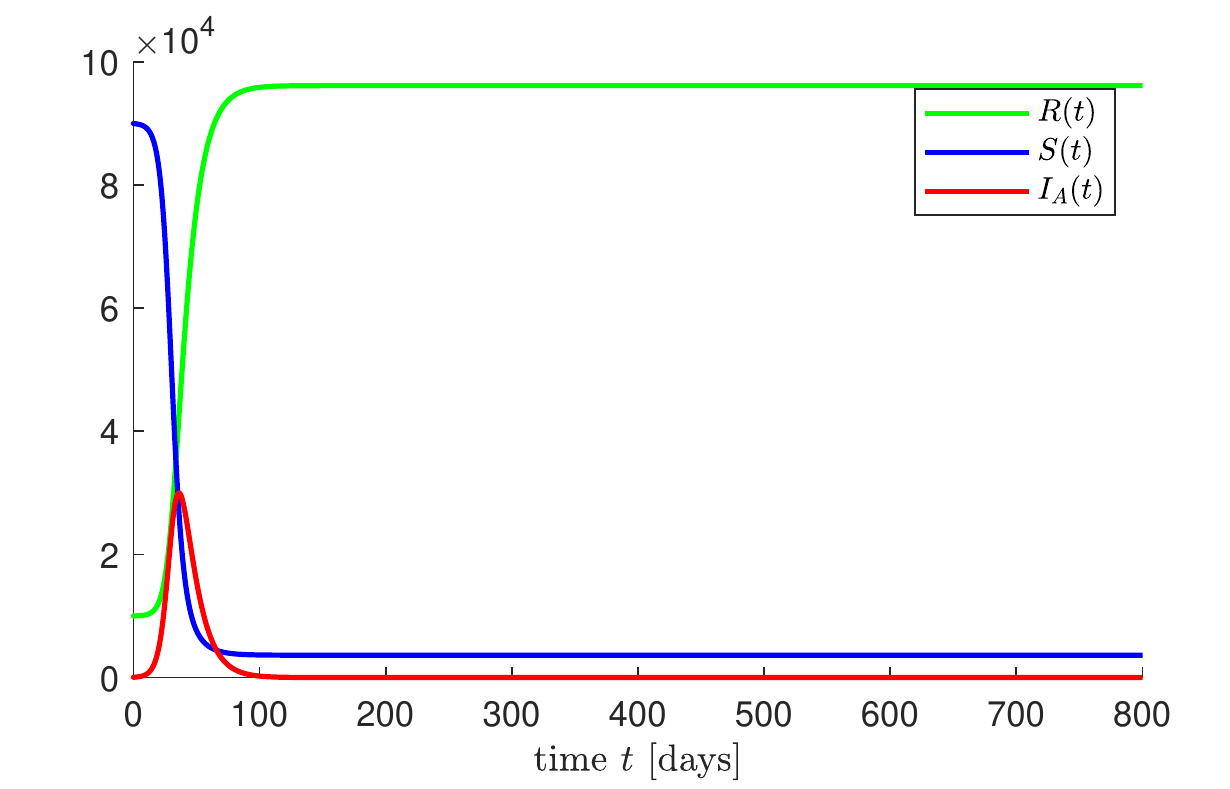}
  \caption{Susceptible, recovered and asymptomatic infected}
  \end{center}
\end{subfigure}\\
\begin{subfigure}{.9\linewidth}
\begin{center}
  \includegraphics[scale=0.6]{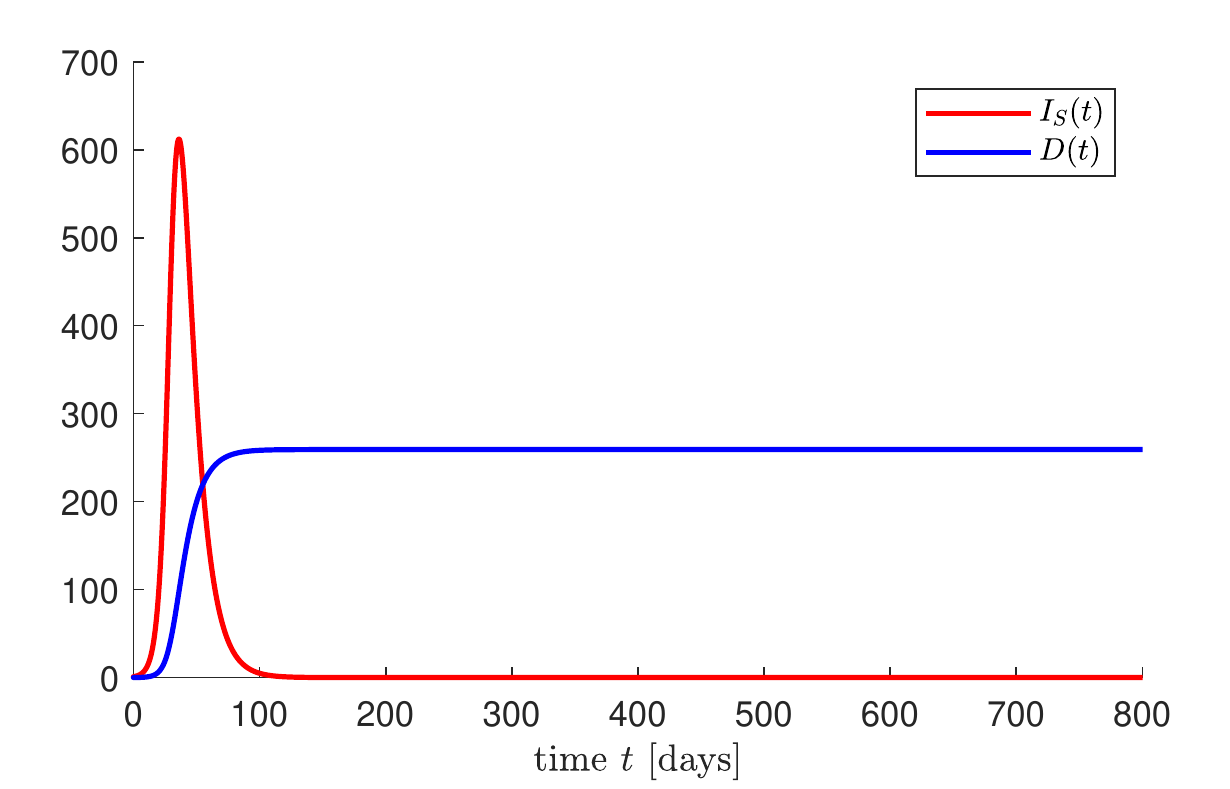}
  \caption{Symptomatic infected and deceased}
  \end{center}
\end{subfigure}
\caption{Simulation of the epidemiological model~\eqref{eq:SIRASD} under $u=0$.}
\label{fig:sim-u=0}
\end{figure}

In contrast to this, Fig.~\ref{fig:sim} shows the same scenario but with social distancing measures enacted according to the control~\eqref{eq:BBFC}. For this simulation we have chosen {$\eps^+ = 2.5 \xi \cdot n_{\rm ICU}$}, which is able to guarantee~\eqref{eq:ICU}. Furthermore, we have chosen two different values for $\eps^-$, namely {$\eps^-_1 = 2 \xi \cdot n_{\rm ICU}$} and {$\eps^-_2 = 5 \xi \cdot n_{\rm ICU}$. We stress that the chosen parameters satisfy assumptions (A1)--(A3), where in particular the lower bound for~$\vp^+$ in~(A3) is $\vp^+ > \max\{\tfrac{M_2}{M_1},M_3\} \approx 23.9$ and hence satisfied. Furthermore, both pairs of controller parameters $(\eps^-_1,\eps^+)$ and $(\eps^-_2,\eps^+)$ satisfy (A4)--(A5). Finally, (A6) is satisfied as well, so that the controller~\eqref{eq:BBFC} is also robust.}

\begin{figure*}[h!]
\begin{subfigure}{\linewidth}
\begin{center}
  \includegraphics[width=0.6\linewidth,trim=0 110 0 0, clip]{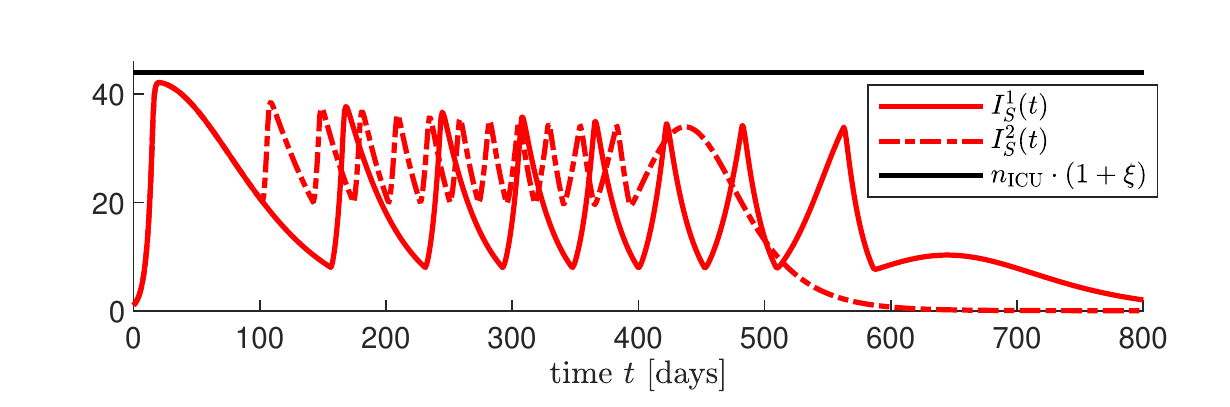}
  \caption{Symptomatic infected and ICU capacity}
  \end{center}
\end{subfigure}\quad
\begin{subfigure}{.48\linewidth}
\begin{center}
  \includegraphics[scale=0.6]{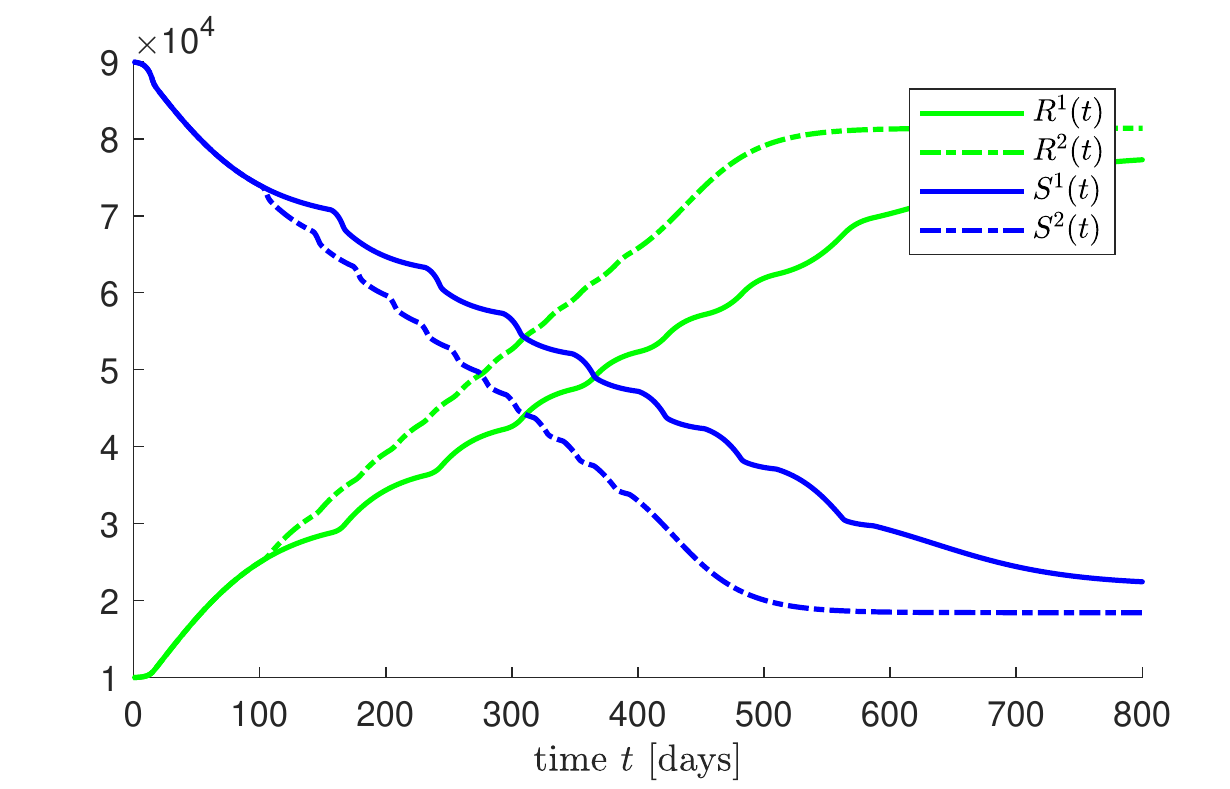}
  \caption{Susceptible and recovered}
  \end{center}
\end{subfigure}
\begin{subfigure}{.48\linewidth}
\begin{center}
  \includegraphics[scale=0.6]{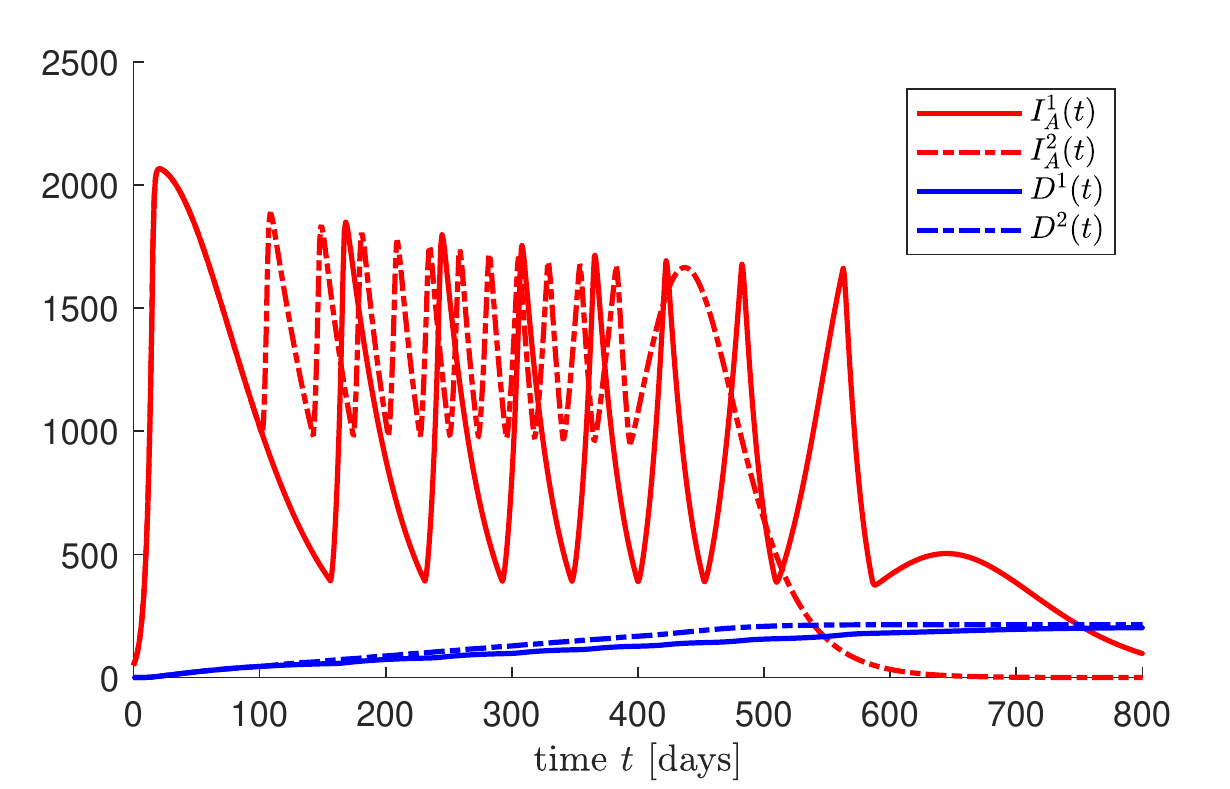}
  \caption{Asymptomatic infected and deceased}
  \end{center}
\end{subfigure}\quad
\begin{subfigure}{.48\linewidth}
\begin{center}
  \includegraphics[scale=0.65]{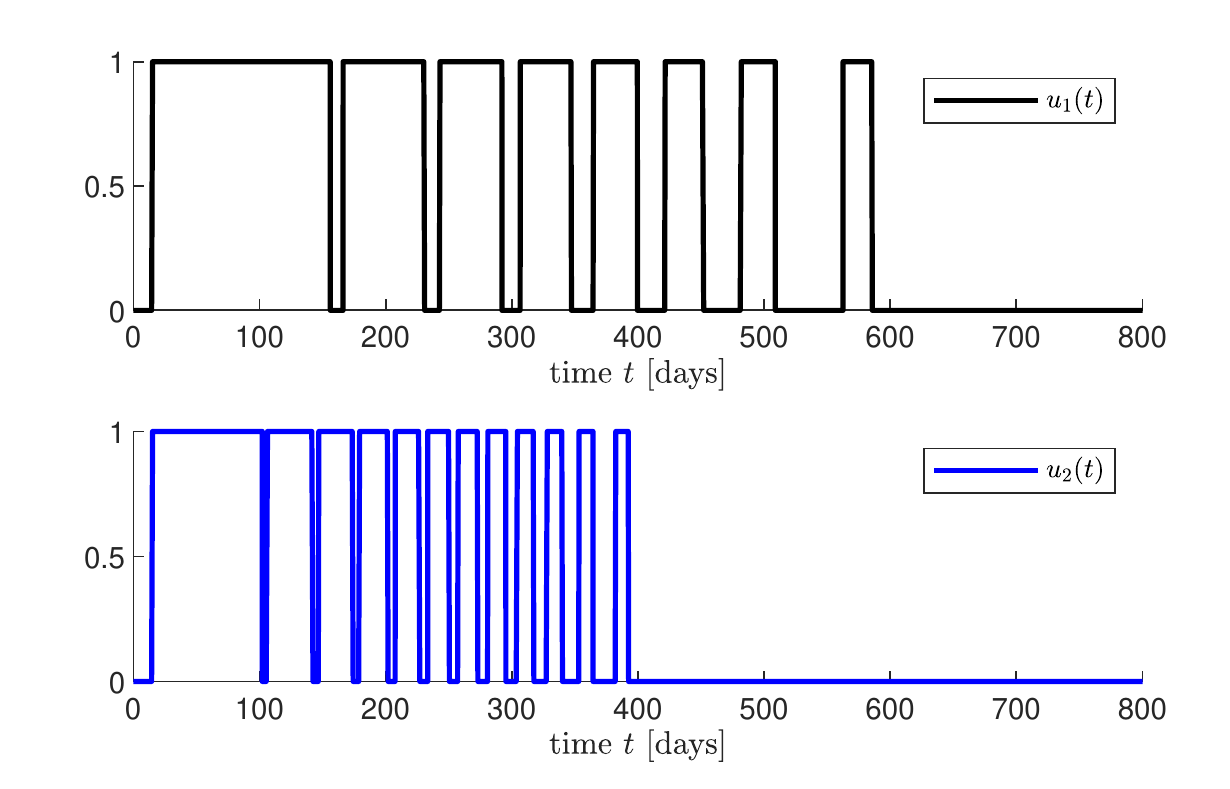}
  \caption{Input functions}
  \end{center}
\end{subfigure}\quad
\begin{subfigure}{.48\linewidth}
\begin{center}
  \includegraphics[scale=0.65]{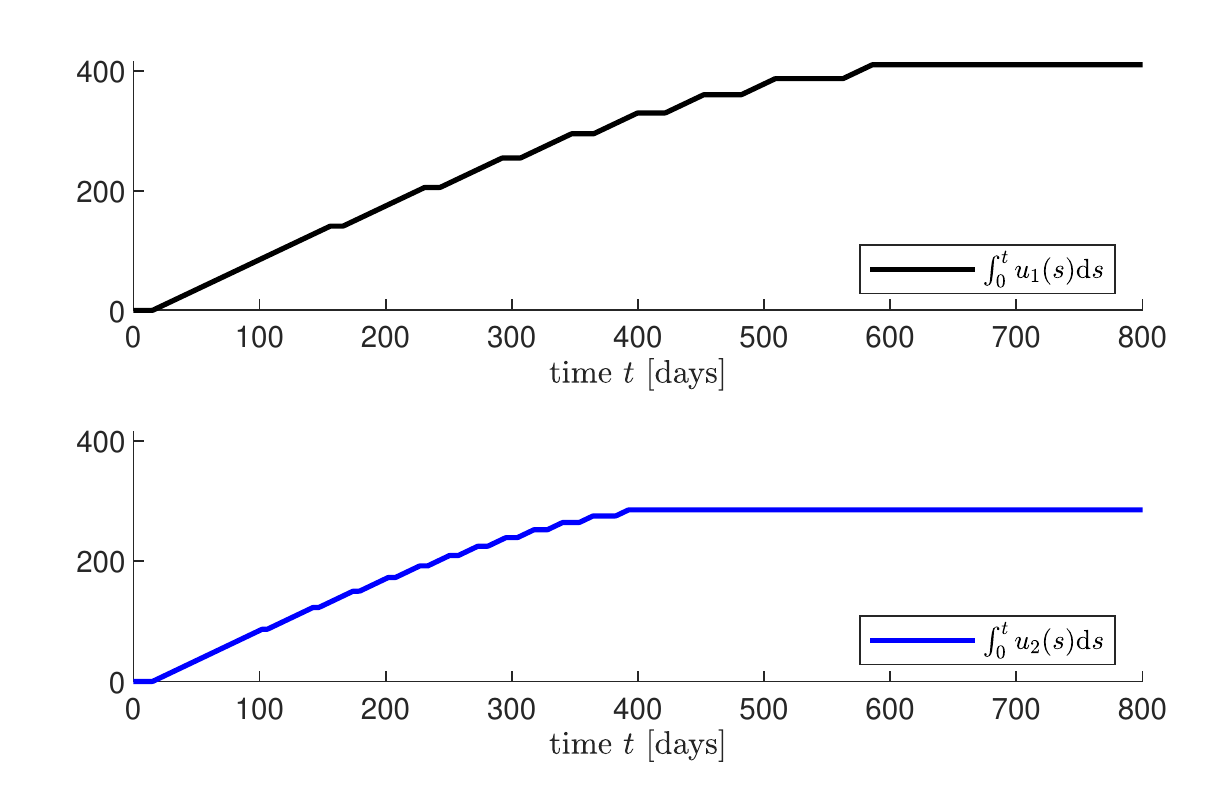}
  \caption{Input costs}
  \end{center}
\end{subfigure}
\caption{Simulation of the epidemiological model~\eqref{eq:SIRASD} under control~\eqref{eq:BBFC} with parameters $\eps^+$ and $\eps^-_i$ for $i=1,2$.}
\label{fig:sim}
\end{figure*}

Similar to various studies before, the simulations depicted in Fig.~\ref{fig:sim} show that social distancing measures are capable of reducing the total number of infected individuals and, as a consequence, the total number of disease induced deaths.  The feedback controller~\eqref{eq:BBFC} is able to guarantee~\eqref{eq:ICU} as shown in Fig.~\ref{fig:sim}\,(a). It can be seen that, while periods with $u(t) = 1$ ensure that the ICU capacity is not exceeded, periods with $u(t) = 0$ (no social distancing) result in quick increases of infection numbers and the switch in the controller~\eqref{eq:BBFC} is triggered again after only a short period of time (at most a few weeks). This shows that lifting social distancing measures over a longer period of time will not be feasible until the end of the pandemic.

As shown in Fig.~\ref{fig:sim}\,(d) the control input $u_2$ corresponding to $\eps^-_2$ exhibits a faster switching pattern than the input~$u_1$ corresponding to~$\eps^-_1$. Several other simulations show that larger values of~$\eps^-$ lead to a faster switching with shorter periods between the switches, but, on the other hand, the pandemic is defeated at an earlier time point (i.e., the time $T>0$ for which $u(t) = 0$ for all $t\ge T$ can be made smaller the larger~$\eps^-$ is chosen). These are two conflicting objectives (few switches vs.\ shorter pandemic) and the {policy makers} have to decide which should be favored; the controller design parameters may be adjusted accordingly.

Another observation reveals that the total number of deaths, i.e., $D_{\max} = \lim_{t\to\infty} D(t)$ depends on the choice of~$\eps^-$. Since minimizing~$D_{\max}$ seems a reasonable goal we have performed a set of simulations so that this is achieved, which led to the above value of $\eps^- = \eps^-_1$. As shown in Fig.~\ref{fig:sim}\,(c) the number of deaths is indeed higher for $\eps^-_2$ and, as shown in Fig.~\ref{fig:sim}\,(b), the total number of infected (represented by $N-R(0)-S(t)$) is larger for~$\eps^-_2$.

{Fig.~\ref{fig:sim}\,(e) shows the input costs associated with the control signals~$u_1$ and~$u_2$ in terms of their $L^1$-norm on the interval $[0,t]$ for $t\ge 0$. It can be seen that these costs are much higher for~$u_1$. {Summarizing, a simple rule of thumb seems to be: Trying to keep the number of symptomatic infected as high as possible without exceeding the ICU capacity seems to achieve the shortest possible course of the pandemic with the smallest input costs, to the detriment of fast switching and a higher number of deaths.} However, the ``total costs'' of the control strategy cannot be assessed by solely considering the input costs, but the total duration of the pandemic and the total number of deaths (as mentioned above) must also be taken into account when defining a suitable total cost function~-- this is a topic of future research which should also involve objectives defined by policy makers.}

Finally, by way of comparison, we like to note that effective control methods for the COVID-19 pandemic based on model-predictive control (MPC) have been developed in~\cite{KohlSchw20,MoraBast20}. However, MPC requires accurate model data and measurements of all state variables for feasibility. As shown in~\cite{KohlSchw20}, uncertainty in the data and measurements can be compensated to a certain extent by using e.g.\ interval predictions, however it is not possible to prove \textit{a priori} that MPC does not exceed the available ICU capacity. This is one of the advantages of the approach presented here.

\section{Conclusion}

We have presented a novel feedback controller which may serve as a decision making mechanism for {policy makers} during the COVID-19 pandemic. The controller is based on the bang-bang funnel controller from~\cite{LibeTren13b} and {robust with respect to uncertainties in the parameters} of the underlying epidemiological model. {It only requires} the measurement of the number of individuals with moderate to severe symptoms, and it is able to keep this number below a threshold determined by the ICU capacity at any time. {Furthermore, the interval of time between successive switching can be influenced by the choice of the controller design parameters.}

Simulations illustrate that the proposed controller~\eqref{eq:BBFC} is able to achieve the control objective and that a relaxation of social distancing policies may quickly lead to increasing infection numbers. Although a relaxation of the distancing measures over a period of only 1--2 weeks may seem pointless, the simulations suggest that a temporary increase of infection numbers, while preventing that the ICU capacity is exceeded, ensures a less extended time course of the COVID-19 pandemic compared to a strict lockdown. At the same time this allows a resumption of social and economic activities during these periods. By a continuous improvement of the available data and the simulations it should be possible to obtain better forecasts of when the input signal switches, which would allow the people to prepare for possible measures or relaxation.

Although the results presented in this work are quite promising regarding automated signals for social distancing measures, this is only a first step towards a universal technique. Future research needs to focus on methods which incorporate different levels of social distancing measures compared to only the two levels (full measures or no measures at all) considered in the present paper. A balanced use of such regulations is more practicable and will be accepted by a wider public. Typical examples are cancelation of big events, carrying face masks in supermarkets and public transport or working from home when possible; these measures may be included in the model by additional control values $u_i \in (0,1)$ as suggested e.g.\ in~\cite{KohlSchw20}.

Another topic of future research is the combination of different models for different countries or regions, where different control values (due to government policies) are active. In particular, it needs to be investigated how the interactions between different regions, based on migration flows, influence the spread of the disease. Such an approach will possibly reveal which social distancing measures must be taken in neighboring regions with different outbreak levels.

Last but not least, we like to note that the approach presented here is not restricted to the model~\eqref{eq:SIRASD} or to the COVID-19 pandemic specifically, but the feedback controller~\eqref{eq:BBFC} can be applied, \textit{mutatis mutandis}, to any epidemiological model available in the literature (modeling any epidemic), appended by the dynamics for the population response.



\appendix

\section{Properties of solutions}

\renewcommand\appendixname{\!}

\begin{lemma}\label{lem:pos_sln}
{Assume that $(S,I_A,I_S,R,D,\psi):[0,\omega)\to \R^6$ is a solution of~\eqref{eq:SIRASD} for non-negative initial conditions and some piecewise constant function $u:[0,\omega)\to \{0,1\}$, where $\omega\in(0,\infty]$. Then
\begin{enumerate}
  \item $S(t), I_A(t), I_S(t), R(t)$ and $D(t)$  are non-negative for all $t\in[0,\omega)$,
  \item if~(A1) and~(A2) hold, then $I_A(t)\ge \tfrac{1-p}{p} I_S(t)$ for all $t\in[0,\omega)$,
  \item if~(A2) holds, then  $S(t) \ge S_{\min}$ for all $t\in[0,\omega)$, where $S_{\min}$ is defined in~\eqref{eq:constants},
  \item $I_A(t) \le \max\{I_A(0)/I_S(0), (1-p)\beta_S/B\} I_S(t)$ for all $t\in[0,\omega)$, where $B$ is defined in~\eqref{eq:constants}.
\end{enumerate} }
\end{lemma}
\begin{proof}
{The proof of~(i) relies on the well-known concept of integrating factors, see e.g.~\cite[Thm.~1]{BakaNwag14}. To show that $S$ is non-negative, let $z(t):= (\beta_A I_A(t) + \beta_S I_S(t)) \tfrac{\psi(t)}{N-D(t)}$ and observe that
\[
  \ddt \left( S(t) e^{z(t)}\right) = \big(\dot S(t) + z(t) S(t)\big) e^{z(t)} = 0,
\]
thus $S(t) = e^{z(0)-z(t)} S(0) \ge 0$ for all $t\in [0,\omega)$. Similarly, it can be shown that $I_A, I_S, R$ and $D$ are non-negative.}\\
{To show~(ii), set $z(t):= I_A(t) - \frac{1-p}{p} I_S(t)$, then
\begin{align*}
    \dot z(t) &= -\alpha_A I_A(t) + \frac{1-p}{p} \frac{\alpha_S}{1-\rho} I_S(t) \\
    &= -\alpha_A z(t) + \frac{1-p}{p} \left(\frac{\alpha_S}{1-\rho} -\alpha_A\right) I_S(t) \stackrel{\rm (A1)}{\ge} -\alpha_A z(t)
\end{align*}
for almost all $t\in[0,\omega)$, and hence Gr\"onwall's lemma implies $z(t) \ge e^{-\alpha_A t} z(0)$. Therefore, we have $I_A(t) \ge e^{-\alpha_A t} z(0) + \frac{1-p}{p} I_S(t) \ge \frac{1-p}{p} I_S(t)$, where the last inequality follows from the assumption $I_A(0) \ge \frac{1-p}{p} I_S(0)$ in~(A2).}\\
{To show~(iii), set $\beta_{\max} := \max\{\beta_S, \beta_A\}$ and $\alpha_{\min} := \min\{\alpha_S,\alpha_A\}$ and recall $N-D(t) \ge R(t) \ge R^0$, then
\begin{align*}
  \dot S(t) &= - \big(\beta_A \psi(t) I_A(t) + \beta_S \psi(t) I_S(t) \big) \frac{S(t)}{N-D(t)}\\
   &\stackrel{\eqref{eq:bound-psi}}{\ge} - \big(\beta_A I_A(t) + \beta_S I_S(t) \big) \frac{S(t)}{R^0} \\
  &\ge -  {\frac{\beta_{\max}}{\alpha_{\min}}} \left( \alpha_A I_A(t) + \alpha_S I_S(t) \right) \frac{S(t)}{R^0} \\
  &= - {\frac{\beta_{\max}}{\alpha_{\min}}} \frac{\dot R(t)}{R^0} S(t)
\end{align*}
for almost all $t\in[0,\omega)$. Then  Gr\"onwall's lemma implies that, for all $t\in[0,\omega)$,
\begin{align*}
  S(t) &\ge S^0 \exp \left( -  {\frac{\beta_{\max}}{\alpha_{\min}}} \int_0^t \frac{\dot R(s)}{R^0} \, {\rm d}s\right) = S^0 e^{ -  \frac{{\beta_{\max}}(R(t)-R^0)}{{\alpha_{\min}} R^0}} \\
  &\ge  S^0 e^{ -  \frac{{\beta_{\max}} (N-R^0)}{ {\alpha_{\min}} R^0}} = S_{\min}.
\end{align*}}\\
{We show~(iv). Set $z(t):=\tfrac{I_A(t)}{I_S(t)}$ and $y(t):=\tfrac{\psi(t) S(t)}{N-D(t)}$, then we may observe that
\begin{align*}
  \dot z(t) &= \left( -p\beta_A z(t)^2 \!+\! ((1\!-\!p)\beta_A\!-\!p\beta_S) z(t) \!+\! (1\!-\!p) \beta_S\right) y(t)\\
   &\quad + \left(\tfrac{\alpha_S}{1-\rho}-\alpha_A\right) z(t)\\
  &= \Big( \Big( -p\beta_A z(t) + ((1-p)\beta_A-p\beta_S)\\
   &\quad + \tfrac{1}{y(t)} \left(\tfrac{\alpha_S}{1-\rho}-\alpha_A\right) \Big) z(t) + (1-p) \beta_S\Big)y(t).
\end{align*}
Therefore, for any $C>0$ it follows that
\begin{align*}
    z(t) &\le \max\Big\{z(0), \frac{(1-p) \beta_S}{C},\\
  &  \quad \frac{1}{p\beta_A}\left(C+(1-p)\beta_A-p\beta_S+\tfrac{1}{y(t)} \left(\tfrac{\alpha_S}{1-\rho}-\alpha_A\right)\right)\Big\}
\end{align*}
for all $t\in[0,\omega)$. Since for $C=B$ we find that
\begin{align*}
    &\frac{1}{p\beta_A}\left(B+(1-p)\beta_A-p\beta_S+\tfrac{1}{y(t)} \left(\tfrac{\alpha_S}{1-\rho}-\alpha_A\right)\right) \\
    &\le  \frac{B+A}{p\beta_A} = \frac{(1-p) \beta_S}{B}
\end{align*}
for~$A$ as in~\eqref{eq:constants}, the claim follows.}
\end{proof}

\bibliographystyle{elsarticle-harv}

\end{document}